\documentclass{amsart}
\usepackage{amssymb}
\usepackage{amsmath}
\usepackage{amsfonts}
\usepackage{geometry}

\newtheorem{theorem}{Theorem}
\theoremstyle{plain}
\newtheorem{acknowledgement}{Acknowledgement}

\newtheorem{condition}{Condition}

\newtheorem{definition}{Definition}
\newtheorem{example}{Example}

\newtheorem{lemma}{Lemma}

\newtheorem{proposition}{Proposition}
\newtheorem{remark}{Remark}

\numberwithin{equation}{section}
\input{tcilatex}

\begin{document}
\title[Non-linear evolution equations]{Non-linear evolution equations driven
by rough paths}
\author{Thomas Cass, Zhongmin Qian, Jan Tudor}

\begin{abstract}
We prove existence and uniqueness results for (mild) solutions to some
non-linear parabolic evolution equations with a rough forcing term. Our
method of proof relies on a careful exploitation of the interplay between
the spatial and time regularity of the solution by capitialising some of
Kato's ideas in semigroup theory. Classical Young integration theory is then
shown to provide a means of interpreting the equation. As an application we
consider the three dimensional Navier-Stokes system with a stochastic
forcing term arising from a fractional Brownian motion with $h>1/2$.
\end{abstract}

\maketitle

\section{Introduction}

In this paper we study the initial-value problem of the following non-linear
evolution equation of parabolic type 
\begin{equation}
\frac{d}{dt}u+Au+\tilde{Q}(u)=F(u)\dot{w}  \label{ev-01}
\end{equation}
in a separable Hilbert space $X$ with an initial value $u_{0}\in X$, where $%
-A$ is the infinitesimal generator of an analytic semigroup $\left\{
P_{t}\right\} _{t\geq0}$, $\tilde{Q}:D(\tilde{Q})\subset X\rightarrow X$ is
a non-linear operator, and $F:X\rightarrow L_{2}(Z,X)$ satisfies a Lipschitz
type condition which we specify later. The driving path for the perturbation 
$w=(w_{t})_{t\geq0}$ is an $\alpha$-H\"{o}lder continuous path in a
separable Hilbert space $Z$, where $\alpha\in(\frac{1}{2},1]$. The class of
paths includes sample paths of fractional Brownian Motion with Hurst
parameter $h>\frac{1}{2}$.

Our motivation for studying equations of this form is the three dimensional
Navier-Stokes system

\begin{eqnarray*}
\frac{\partial }{\partial t}u+u.\nabla u &=&\Delta u-\nabla p+F(u)\dot{w} \\
\nabla \cdot u &=&0
\end{eqnarray*}%
in a bounded domain $\Omega $ with compact, smooth boundary $\Gamma $. The
force term $F(u)\dot{w}$ is modeled by a fractional Brownian Motion in the
Hilbert space $Z$, and $u_{t}$ describes the velocity of the fluid flow
under the influence of a stochastic force $F(u)\dot{w}$ \ and subject to the
no-slip condition. By projecting to the $L^{2}$-space of solenoidal vector
fields on $\Omega $ via the\ Leray-Hopf projection $P_{\infty }$, the above
equation can be written as the following evolution equation 
\begin{equation}
\frac{d}{dt}u+Au+P_{\infty }(u.\nabla u)=P_{\infty }F(u)\dot{w}
\label{eq:NS-EV}
\end{equation}%
where $A=-P_{\infty }\circ \Delta $ is the Stokes operator, which is a self
adjoint operator.

If we let $\left\{ P_{t}\right\} _{t\geq 0}$ be the semigroup generated by $%
-A$, $h(s)=P_{t-s}u(s)$ and we consider the formal calculation 
\begin{eqnarray*}
F^{\prime }(s) &=&(AP_{t-s})u(s)+P_{t-s}u^{\prime }(s) \\
&=&(AP_{t-s})u(s)-(P_{t-s}A)u(s) \\
&&-P_{t-s}\tilde{Q}(u(s))+P_{t-s}\left( F(u(s))\dot{w}_{s}\right) \\
&=&-P_{t-s}\tilde{Q}(u(s))+P_{t-s}F(u(s))\dot{w}_{s}\text{,}
\end{eqnarray*}%
then integrating from $0$ to $t$ one obtains 
\begin{equation}
u(t)=P_{t}u_{0}-\int_{0}^{t}P_{t-s}\tilde{Q}(u(s))ds+%
\int_{0}^{t}P_{t-s}F(u(s))\dot{w}_{s}ds\text{.}  \label{mi-1}
\end{equation}%
The first integral on the right-hand side is considered as the Bochner
integral, the second one needs to be interpreted as some sort of stochastic
integration. Hence the previous has to be written as 
\begin{equation*}
u(t)=P_{t}u_{0}-\int_{0}^{t}P_{t-s}\tilde{Q}(u(s))ds+%
\int_{0}^{t}P_{t-s}F(u(s))dw_{s}
\end{equation*}%
so that if $u$ satisfies the above, we call it a mild solution to (\ref%
{ev-01}). For Navier-Stokes equations in the classical case that $w$ is
differentiable Kato made the following observation (see \cite{fujita}) which
proves significant. If $t-s>0$ and $x\in X$, then $P_{t-s}x$ belongs to the
domain $D(A^{\tau })$ for any real $\tau $. On the other hand, if $\tau >0$, 
$A^{-\tau }$ is bounded, the integral involving the non-linear terms of (\ref%
{mi-1}) can be rewritten to enhance the regularity. For example 
\begin{equation*}
\int_{0}^{t}P_{t-s}\tilde{Q}(u(s))ds=\int_{0}^{t}A^{\tau }P_{t-s}A^{-\tau }%
\tilde{Q}(u(s))ds\text{.}
\end{equation*}%
As long as $A^{-\tau }\tilde{Q}(u(s))$ is bounded on $[0,t]$, then, since $%
||A^{\tau }P_{t-s}||\leq C(t-s)^{-\tau }$ which is still integrable on $%
[0,t] $ for any $\tau <1$. Therefore, we may consider the non-linear
operator $Q=A^{-\tau }\tilde{Q}$ instead of $\tilde{Q}$, where $%
Q(x)=A^{-\tau }\tilde{Q}(x)$ for $x\in D(Q)$, but in general the domain of
definition $D(Q)$ can be extended to be a little bit larger than $D(\tilde{Q}%
)$ due to the fact that $A^{-\tau }$ is a bounded linear operator.
Essentially the same idea applies to the stochastic case, although the
difficulty is in the stochastic term.

Now, according to Sobolevski \cite{sobolev}, we have the estimate 
\begin{equation}
||A^{-\frac{1}{4}}P_{\infty }(w.\nabla u)||\leq C||\sqrt{A}u||||\sqrt{A}w||
\label{eq:sobiq}
\end{equation}%
where $||\cdot ||$ is the $L^{2}$-norm and $A$ is the Stokes operator (with
Dirichlet condition) in a bounded domain with smooth boundary $\Gamma $.
Hence as we are motivated by the stochastic version of this situation, we
find fixed points in a certain space of the map 
\begin{equation*}
\mathbb{L}u(t)=P_{t}u_{0}-\int_{0}^{t}A^{\tau
}P_{t-s}Q(u(s))ds+\int_{0}^{t}P_{t-s}F(u(s))dw_{s}
\end{equation*}%
where $w$ is a fractional Brownian Motion.

Several existing papers study evolution equations driven by a fractional
Brownian Motion (fBM): we summarize some of the main results here. In \cite%
{maslowski} the authors prove path-wise existence and uniqueness of mild
solutions to equations of the form (\ref{ev-01}) where the perturbation
operator $\tilde{Q}:X\rightarrow X$ is taken to be continuous with linear
growth. Additionally they assume the stochastic force term $F$ satisfies $%
F:X\rightarrow L\left( X\right) $ and the composition $P_{t}F\left( \cdot
\right) $ has linear growth and is Lipschitz with a constant proportional to 
$t^{\gamma _{1}},$ $t^{-\gamma _{2}}$ respectively for $\gamma _{1},\gamma
_{2}\in \left[ 0,1\right) $. The solutions are obtained from fractional
calculus methods for fBM with Hurst parameter $h>\frac{1}{2}$ and are
elements of the space of $X$ valued paths such that the norm 
\begin{equation*}
\left\Vert u\right\Vert :=\sup_{t\in \left[ 0,T\right] }\left\Vert u\left(
t\right) \right\Vert _{X}+\sup_{t\in \left[ 0,T\right] }\int_{0}^{t}\frac{%
\left\Vert u\left( t\right) -u\left( s\right) \right\Vert _{X}}{\left(
t-s\right) ^{1+\alpha }}ds,
\end{equation*}%
$\alpha \in \left( 1-h,\frac{1}{2}\right) ,$ is finite. Both of the cases $h>%
\frac{1}{2}$ and $h<\frac{1}{2}$ are treated in \cite{tindel} where
necessary and sufficient conditions are obtained for existence of solutions
to stochastically forced linear evolution equations with linear noise, i.e.
of the form 
\begin{equation}
\frac{d}{dt}u+Au=F\dot{w}  \label{lineqn}
\end{equation}%
where $F$ does not depend on $u$. Here the authors do not solve path-wise
but instead use Skorohod type integration to produce a solution which is a
square integrable (in the stochastic sense) $X$ valued process and then show
space regularity of the solution as a continuous map into the domain of $-A$%
. More recently, in \cite{gubinelli}, the equation (\ref{lineqn}) with
nonlinear $F$ is studied (i.e. $F:X\rightarrow L_{2}\left( Z,X\right) $ with
some Lipschitz type conditions) and path-wise mild solutions are obtained
using Young integration. In this paper we treat a non-linear evolution
equation forced by non-linear noise. The ideas used are straight forward
while the difficulty lies in the delicate interplay between the space and
time regularity of the paths involved.

The paper is organized as follows. In the second section, we set out the
conditions on the operators involved and define the space in which we seek a
mild solution. In section 3, we make some initial estimates which follow
from our conditions and are used in the fixed point proof; in the fourth and
fifth sections we discuss the Bochner and Young integral terms of our fixed
point map. Subsequently we prove the existence of a mild solution and
finally discuss the application to the Navier-Stokes equations.

\section{Preliminaries}

We would like to set up the technical conditions on various terms appearing
in the previous evolution equation. The model example is randomly enforced
Navier-Stokes equation in two or three dimensions, in which $A=-P_{\infty
}\circ \Delta $ (or $A=-P_{\infty }\circ \Delta +I$ where $I$ is the
identity operator) is the\ Stokes operator on a bounded domain $\Omega $
together with Dirichlet boundary conditions, and $\tilde{Q}(u)=P_{\infty
}\left( u.\nabla u\right) $. Therefore our assumptions will be motivated
with the aim of applying the abstract setting to this case.

With this in mind, we assume $A$ is a positive-definite, self-adjoint
operator on $X$, with positive spectral gap, so that the spectrum of $A$
lies in the half line $[\lambda _{0},\infty )$. Its domain is denoted by $%
D(A)$. For every real $\epsilon $, $A^{\epsilon }$ is again a self-adjoint
operator, and $A^{\epsilon }$ is bounded if $\epsilon <0$. The domain $%
D(A^{\epsilon })$ is decreasing, so that $D(A)\subset D(A^{\epsilon })$ for $%
\epsilon \in \lbrack 0,1]$ and $D(A^{\epsilon })$ is a Hilbert space under
the norm $||A^{\epsilon }x||$.

Let $P_{t}=e^{-tA}$ be the $C_{0}$-semigroup of contractions on $X$
generated by $-A$. We will frequently use the following facts: for any $t>0$
and $x\in X$, $P_{t}x\in D\left(A^{\epsilon}\right)$ for any $\epsilon\leq1$
and $||A^{\epsilon}P_{t}||\leq\frac{C}{t^{\epsilon}}$ for some constant $C$
depending on $\epsilon$.

For every $t>0$ and $x\in D(A^{\epsilon })$ for some $\epsilon \in (0,1]$,
we have 
\begin{eqnarray*}
x-P_{t}x &=&\int_{0}^{t}AP_{s}xds \\
&=&\int_{0}^{t}A^{1-\epsilon }P_{s}A^{\epsilon }xds
\end{eqnarray*}%
so that 
\begin{eqnarray}
||x-P_{t}x|| &\leq &||A^{\epsilon }x||\int_{0}^{t}||A^{1-\epsilon }P_{s}||ds
\notag \\
&\leq &C(\epsilon )||A^{\epsilon }x||\int_{0}^{t}s^{-(1-\epsilon )}ds  \notag
\\
&=&\frac{C(\epsilon )}{\epsilon }t^{\epsilon }||A^{\epsilon }x||
\label{regularity}
\end{eqnarray}%
an estimate which will be needed in order to derive useful a priori
estimates. Now we are in a position to formulate the technical conditions on
the non-linear term $Q$ as follows.

\begin{condition}
\label{Qcond}Let $\tau ,\delta \in \lbrack 0,1)$, $\delta \leq \alpha \in (%
\frac{1}{2},1]$ be three parameters such that $\delta +\tau <1$, $\alpha
+\tau <1$, $\alpha +\delta >1$, and let $K_{j}:[0,\infty )\rightarrow
\lbrack 0,\infty )$ ($j=0,1,2$) be three increasing functions.

\begin{enumerate}
\item The domain of definition of $Q$, $D(Q)=D(A^{\delta })$.

\item $Q$ is bounded on $D(A^{\delta })$, and $||Q(x)||\leq
K_{2}(||A^{\delta }x||)$.

\item $Q$ is weakly Gateaux differentiable on $D(A^{\delta })$: if $x\in
D(A^{\delta })$, then there is a linear operator $DQ(x):D(A^{\delta
})\rightarrow X$, such that $\varepsilon \rightarrow \langle Q(x+\varepsilon
\xi ),z\rangle $ is differentiable for any $\xi \in D(A^{\delta })$, $z\in
X^{\ast }$, 
\begin{equation*}
\langle DQ(x)\xi ,z\rangle =\left. \frac{d}{d\varepsilon }\right\vert
_{\varepsilon =0}\langle H(x+\varepsilon \xi ),z\rangle
\end{equation*}%
and $||DQ(x)\xi ||\leq K_{1}(||A^{\delta }x||)||A^{\delta }\xi ||$. That is, 
$DQ(x)$ is a bounded linear operator in $D(A^{\delta })$.
\end{enumerate}
\end{condition}

\begin{example}
If $A=-P_{\infty}\circ\Delta$ is the Stokes operator with Dirichlet boundary
condition on a bounded domain with smooth boundary $\Gamma$, and $Q(u)=A^{-%
\frac{1}{4}}P_{\infty}(u.\nabla u)$, on the Hilbert space $K_{2}(\Omega)$,
then we can choose $\delta=\frac{1}{2}$, $\tau=\frac{1}{4}$.
\end{example}

\bigskip Now let us set out the conditions on the non-linear operator $%
F:X\rightarrow L_{2}(Z,X)$.

\begin{condition}
\label{Fcond}We assume the following for $\varepsilon =\max \left\{ \alpha
+\delta ,2\alpha \right\} $.

\begin{enumerate}
\item For every $x\in X,$ $\xi \in Z$ we have $F\left( x\right) \xi \in
D\left( A^{\varepsilon }\right) $ and $A^{\varepsilon }F$ is globally
Lipschitz in the sense that 
\begin{equation}
||A^{\varepsilon }\left[ F(x)-F\left( y\right) \right] ||_{L_{2}\left(
Z,X\right) }\leq C||x-y||_{X}\mbox{ }\forall x,y\in X.  \label{eq:cond1}
\end{equation}

\item The composition $A^{\delta +\varepsilon }F$ is globally relatively
Lipschitz in the sense that 
\begin{equation}
||A^{\delta +\varepsilon }\left[ F(x)-F\left( y\right) \right]
||_{L_{2}\left( Z,X\right) }\leq C||A^{\delta }\left[ x-y\right] ||_{X}\text{
}\forall x,y\in D\left( A^{\delta }\right) \text{.}  \label{eq:cond2}
\end{equation}

\item Additionally, we assume that $A^{\varepsilon }F:X\rightarrow
L_{2}\left( Z,X\right) $ is weakly Gateaux differentiable, i.e. $\tau
\rightarrow \left\langle \left\langle A^{\varepsilon }F\left( x+\tau
y\right) ,\xi \right\rangle \right\rangle $ is differentiable and there
exists a linear operator $DA^{\varepsilon }F\left( x\right) :X\rightarrow
L_{2}\left( Z,X\right) $ such that $\left. \frac{d}{d\tau }\left\langle
\left\langle A^{\varepsilon }F\left( x+\tau y\right) ,\xi \right\rangle
\right\rangle \right\vert _{\tau =0}=\left\langle \left\langle
DA^{\varepsilon }F\left( x\right) y,\xi \right\rangle \right\rangle $ for
all $\xi \in L_{2}\left( Z,X\right) ^{\prime }$, and satisfies the bounds 
\begin{equation*}
\sup_{x}\left\Vert DA^{\varepsilon }F\left( x\right) \right\Vert _{L\left(
X,L_{2}\left( Z,X\right) \right) }<\infty
\end{equation*}%
and 
\begin{equation*}
\left\Vert DA^{\varepsilon }F\left( x\right) -DA^{\varepsilon }F\left(
y\right) \right\Vert _{L\left( X,L_{2}\left( Z,X\right) \right) }\leq
C\left\Vert x-y\right\Vert .
\end{equation*}
\end{enumerate}
\end{condition}

Throughout this paper, the non-negative constants $C_{j}$ may be different
from place to place, and $C_{j}$ may depend on only the parameters $\tau $, $%
\delta $, $\alpha $ and the operator $A$, but are independent of $u$, $w$,
and $T$.

\section{Initial Estimates}

The main goal of this section is to derive technical estimates which are
used to show the existence and uniqueness of solutions to the evolution
equation (\ref{ev-01}). Let us define $\mathbb{H}_{T}$ to be the collection
of all $\alpha $-H\"{o}lder continuous paths $u=(u_{t})$ in $X$ such that
for each $t\in \lbrack 0,T]$, $u_{t}\in D(A^{\delta })$ and $t\rightarrow
A^{\delta }u_{t}$, $(s,t)\rightarrow \frac{u_{t}-u_{s}}{|t-s|^{\alpha }}$
are bounded. If $u\in \mathbb{H}_{T}$ we define the norm 
\begin{equation}
||u||_{\mathbb{H}_{T}}=\sup_{t\in \lbrack 0,T]}||A^{\delta }u_{t}||+\sup 
_{\substack{ s,t\in \lbrack 0,T]  \\ s\neq t}}\frac{||u_{t}-u_{s}||}{%
|t-s|^{\alpha }}\text{.}  \label{c1-a}
\end{equation}%
By definition, if $||u||_{\mathbb{H}_{T}}\leq \beta $ then 
\begin{equation*}
||A^{\delta }u_{t}||\leq \beta \text{, \ \ }||u_{t}-u_{s}||\leq \beta
\left\vert t-s\right\vert ^{\alpha }
\end{equation*}%
for any $s,t\in \lbrack 0,T]$. And since $||x||\leq C_{\delta }||A^{\delta
}x||$ we also have $||u_{t}||\leq C_{\delta }\beta $. It is also obvious
that $(\mathbb{H}_{T},||\cdot ||_{\mathbb{H}_{T}})$ is a Banach space.

\begin{lemma}
\label{lip-0}$Q$ is locally Lipschitz continuous with respect to $A^{\delta
} $: 
\begin{equation}
||Q(x)-Q(y)||\leq K_{1}(||A^{\delta }x||+||A^{\delta }y||)||A^{\delta
}(x-y)||\text{.}  \label{mi-3}
\end{equation}
\end{lemma}

\begin{proof}
Let $\theta (s)=sx+(1-s)y$. Then $\theta \in C^{1}([0,1],X)$, $\theta
^{\prime }(s)\in D(A^{\delta })$ and $||A^{\delta }\theta (s)||\leq
||A^{\delta }x||+||A^{\delta }y||$. For every $\xi \in X^{\ast }$ 
\begin{eqnarray*}
\langle Q(x)-Q(y),\xi \rangle &=&\int_{0}^{1}\frac{d}{ds}\langle Q(\theta
(s)),\xi \rangle \\
&=&\int_{0}^{1}\langle DQ(\theta (s))\theta ^{\prime }(s),\xi \rangle
\end{eqnarray*}%
which yields that 
\begin{eqnarray*}
\left\vert \langle Q(x)-Q(y),\xi \rangle \right\vert &\leq
&\int_{0}^{1}||DQ(\theta (s))\theta ^{\prime }(s)||||\xi || \\
&\leq &\int_{0}^{1}K_{1}(||A^{\delta }x||+||A^{\delta }y||)||A^{\delta
}\theta ^{\prime }(s)||||\xi ||ds
\end{eqnarray*}%
so (\ref{mi-3}) follows.
\end{proof}

\bigskip

We turn to the operator $F.$

\begin{lemma}
\label{Flemma}If $F$ satisfies Condition \ref{Fcond} then we have for every $%
x,y\in X$ and $\xi \in Z$ 
\begin{equation}
\left\vert \left\vert A^{\epsilon }\left[ F\left( x\right) -F\left( y\right) %
\right] \xi \right\vert \right\vert _{X}\leq C\left\vert \left\vert
x-y\right\vert \right\vert _{X}\left\vert \left\vert \xi \right\vert
\right\vert _{Z}  \label{consequence1}
\end{equation}

\noindent and for all $x\in D\left( A^{\delta }\right) $, $\xi \in Z$%
\begin{equation}
\left\vert \left\vert A^{\delta +\epsilon }F\left( x\right) \xi \right\vert
\right\vert _{X}\leq C_{1}\left\vert \left\vert A^{\delta }x\right\vert
\right\vert \left\vert \left\vert \xi \right\vert \right\vert
_{Z}+C_{1}\left\vert \left\vert \xi \right\vert \right\vert _{Z}.
\label{consequence2}
\end{equation}
\end{lemma}

\begin{proof}
Since the operator norm is bounded above by the Hilbert-Schmidt norm it
follows immediately from (\ref{eq:cond1}) that%
\begin{equation*}
\left\vert \left\vert A^{\epsilon }\left[ F\left( x\right) -F\left( y\right) %
\right] \xi \right\vert \right\vert _{X}\leq C\left\vert \left\vert
A^{\epsilon }\left[ F\left( x\right) -F\left( y\right) \right] \right\vert
\right\vert _{L_{2}\left( Z,X\right) }\left\vert \left\vert \xi \right\vert
\right\vert _{Z}.
\end{equation*}

\noindent Similarly, using (\ref{eq:cond2}) we have for any $y\in X$%
\begin{eqnarray*}
\left\vert \left\vert A^{\delta +\epsilon }F\left( x\right) \xi \right\vert
\right\vert _{X} &\leq &\left\vert \left\vert A^{\delta +\epsilon }\left[
F\left( x\right) -F\left( y\right) \right] \xi \right\vert \right\vert
_{X}+\left\vert \left\vert A^{\delta +\epsilon }F\left( y\right) \xi
\right\vert \right\vert _{X} \\
&\leq &C\left\vert \left\vert A^{\delta }x\right\vert \right\vert
_{X}\left\vert \left\vert \xi \right\vert \right\vert +C\left\vert
\left\vert A^{\delta }y\right\vert \right\vert _{X}\left\vert \left\vert \xi
\right\vert \right\vert +\left\vert \left\vert A^{\delta +\epsilon }F\left(
y\right) \right\vert \right\vert _{L_{2}\left( Z,X\right) }\left\vert
\left\vert \xi \right\vert \right\vert \\
&\leq &C_{1}\left\vert \left\vert A^{\delta }x\right\vert \right\vert
_{X}\left\vert \left\vert \xi \right\vert \right\vert +C_{2}.
\end{eqnarray*}
\end{proof}

\begin{lemma}
Suppose that $F:X\rightarrow L_{2}\left( Z,X\right) $ satisfies Condition %
\ref{Fcond}. If $u,v\in \mathbb{H}_{T}$ we have for $0\leq s<t\leq T$ 
\begin{eqnarray*}
&&\left\vert \left\vert \left[ A^{\epsilon }F\left( u_{t}\right)
-A^{\epsilon }F\left( u_{s}\right) \right] -\left[ A^{\epsilon }F\left(
v_{t}\right) -A^{\epsilon }F\left( v_{s}\right) \right] \right\vert
\right\vert _{L_{2}\left( Z,X\right) } \\
&\leq &C\left( \sup_{x\in X}\left\vert \left\vert DA^{\epsilon }F\left(
x\right) \right\vert \right\vert +\frac{3}{2}\left\vert \left\vert
v\right\vert \right\vert _{\mathbb{H}_{T}}\right) \left\vert \left\vert
u-v\right\vert \right\vert _{\mathbb{H}_{T}}\left( t-s\right) ^{\alpha },
\end{eqnarray*}

\noindent which implies, in particular, that $A^{\epsilon }P_{t-\cdot }\left[
F\left( u_{\cdot }\right) -A^{\epsilon }F\left( v_{\cdot }\right) \right] :%
\left[ 0,t\right] \rightarrow L_{2}\left( Z,X\right) $ is $\alpha -$H\"{o}%
lder continuous with 
\begin{equation*}
\left\vert \left\vert A^{\epsilon }P_{t-\cdot }\left[ F\left( u_{\cdot
}\right) -A^{\epsilon }F\left( v_{\cdot }\right) \right] \right\vert
\right\vert _{\alpha -\text{H\"{o}l;}\left[ 0,T\right] }\leq \left(
\sup_{x\in X}\left\vert \left\vert DA^{\epsilon }F\left( x\right)
\right\vert \right\vert +\frac{3}{2}\left\vert \left\vert v\right\vert
\right\vert _{\mathbb{H}_{T}}\right) \left\vert \left\vert u-v\right\vert
\right\vert _{\mathbb{H}_{T}}.
\end{equation*}
\end{lemma}

\begin{remark}
Here as elsewhere we denote the $\alpha -$H\"{o}lder norm of a path $u:\left[
s,t\right] \rightarrow W$ in a Banach space $\left( W,\left\vert \left\vert
\cdot \right\vert \right\vert _{W}\right) $ by 
\begin{equation*}
\left\vert \left\vert u\right\vert \right\vert _{\alpha -\text{H\"{o}l;}%
\left[ s,t\right] }=\sup_{\substack{ t_{1}<t_{2},  \\ t_{1},t_{2}\in \left[
s,t\right] }}\frac{\left\vert \left\vert u\left( t_{2}\right) -u\left(
t_{1}\right) \right\vert \right\vert _{W}}{\left( t_{2}-t_{1}\right)
^{\alpha }}.
\end{equation*}
\end{remark}

\begin{proof}
Let $\theta _{\tau }^{u}=\tau u_{t}+\left( 1-\tau \right) u_{s},$ $\theta
_{\tau }^{v}=\tau v_{t}+\left( 1-\tau \right) v_{s}$ and $\xi \in
L_{2}\left( Z,X\right) ^{\prime }$ then we have 
\begin{eqnarray*}
&&\left\langle \left\langle \left[ A^{\epsilon }F\left( u_{t}\right)
-A^{\epsilon }F\left( u_{s}\right) \right] -\left[ A^{\epsilon }F\left(
v_{t}\right) -A^{\epsilon }F\left( v_{s}\right) \right] ,\xi \right\rangle
\right\rangle \\
&=&\int_{0}^{1}\frac{d}{d\tau }\left\langle \left\langle A^{\epsilon
}F\left( \theta _{\tau }^{u}\right) ,\xi \right\rangle \right\rangle d\tau
-\int_{0}^{1}\frac{d}{d\tau }\left\langle \left\langle A^{\epsilon }F\left(
\theta _{\tau }^{v}\right) ,\xi \right\rangle \right\rangle d\tau \\
&=&\int_{0}^{1}\left\langle \left\langle DA^{\epsilon }F\left( \theta _{\tau
}^{u}\right) \left( \frac{d}{d\tau }\left[ \theta _{\tau }^{u}-\theta _{\tau
}^{v}\right] \right) ,\xi \right\rangle \right\rangle d\tau \\
&&+\int_{0}^{1}\left\langle \left\langle \left[ DA^{\epsilon }F\left( \theta
_{\tau }^{u}\right) -DA^{\epsilon }F\left( \theta _{\tau }^{v}\right) \right]
\frac{d}{d\tau }\theta _{\tau }^{v},\xi \right\rangle \right\rangle d\tau .
\end{eqnarray*}

\noindent This implies 
\begin{eqnarray*}
&&\left\vert \left\langle \left\langle \left[ A^{\epsilon }F\left(
u_{t}\right) -A^{\epsilon }F\left( u_{s}\right) \right] -\left[ A^{\epsilon
}F\left( v_{t}\right) -A^{\epsilon }F\left( v_{s}\right) \right] ,\xi
\right\rangle \right\rangle \right\vert \\
&\leq &\int_{0}^{1}\left\vert \left\vert DA^{\epsilon }F\left( \theta _{\tau
}^{u}\right) \left( \frac{d}{d\tau }\left[ \theta _{\tau }^{u}-\theta _{\tau
}^{v}\right] \right) \right\vert \right\vert \left\vert \left\vert \xi
\right\vert \right\vert d\tau \\
&&+\int_{0}^{1}\left\vert \left\vert \left[ DA^{\epsilon }F\left( \theta
_{\tau }^{u}\right) -DA^{\epsilon }F\left( \theta _{\tau }^{v}\right) \right]
\frac{d}{d\tau }\theta _{\tau }^{v}\right\vert \right\vert \left\vert
\left\vert \xi \right\vert \right\vert d\tau \\
&\leq &\int_{0}^{1}\sup_{x\in X}\left\vert \left\vert DA^{\epsilon }F\left(
x\right) \right\vert \right\vert \left\vert \left\vert \left[ u_{t}-u_{s}%
\right] -\left[ v_{t}-v_{s}\right] \right\vert \right\vert \left\vert
\left\vert \xi \right\vert \right\vert d\tau \\
&&+\int_{0}^{1}\left\vert \left\vert \theta _{\tau }^{u}-\theta _{\tau
}^{v}\right\vert \right\vert \left\vert \left\vert v_{t}-v_{s}\right\vert
\right\vert \left\vert \left\vert \xi \right\vert \right\vert d\tau \\
&=&\int_{0}^{1}\sup_{x\in X}\left\vert \left\vert DA^{\epsilon }F\left(
x\right) \right\vert \right\vert \left\vert \left\vert \left[ u_{t}-u_{s}%
\right] -\left[ v_{t}-v_{s}\right] \right\vert \right\vert \left\vert
\left\vert \xi \right\vert \right\vert d\tau \\
&&+\int_{0}^{1}\left[ \tau \left\vert \left\vert \left( u_{t}-u_{s}\right)
-\left( v_{t}-v_{s}\right) \right\vert \right\vert +\left\vert \left\vert
u_{s}-v_{s}\right\vert \right\vert \right] \left\vert \left\vert
v_{t}-v_{s}\right\vert \right\vert \left\vert \left\vert \xi \right\vert
\right\vert d\tau \\
&\leq &\sup_{x\in X}\left\vert |DA^{\epsilon }F\left( x\right) \right\vert
\left\vert \left\vert u-v\right\vert \right\vert _{\mathbb{H}_{T}}\left(
t-s\right) ^{\alpha }\left\vert \left\vert \xi \right\vert \right\vert \\
&&+\int_{0}^{1}\left[ \tau \left\vert \left\vert \left( u_{t}-u_{s}\right)
-\left( v_{t}-v_{s}\right) \right\vert \right\vert +\left\vert \left\vert
u_{s}-v_{s}\right\vert \right\vert \right] \left\vert \left\vert
v\right\vert \right\vert _{\mathbb{H}_{T}}\left( t-s\right) ^{\alpha
}\left\vert \left\vert \xi \right\vert \right\vert d\tau \\
&\leq &\left( \sup_{x\in X}\left\vert |DA^{\epsilon }F\left( x\right)
\right\vert +\frac{3}{2}\left\vert \left\vert v\right\vert \right\vert _{%
\mathbb{H}_{T}}\right) \left\vert \left\vert u-v\right\vert \right\vert _{%
\mathbb{H}_{T}}\left( t-s\right) ^{\alpha }\left\vert \left\vert \xi
\right\vert \right\vert
\end{eqnarray*}

and the result follows.
\end{proof}

\section{\protect\bigskip The Bochner Integral}

\begin{proposition}
Let $u\in \mathbb{H}_{T}$ such that $\sup_{t\in (0,T]}||A^{\delta
}u_{t}||<\infty $. Let $\varepsilon \in \lbrack 0,1)$. Then

(1) For any $t\in (0,T]$, $s\rightarrow A^{\varepsilon }P_{t-s}Q(u(s))$ is
continuous on $(0,t)$.

(2) For every $s<t\in (0,T]$, $\int_{s}^{t}A^{\varepsilon }P_{t-r}Q(u(r))dr$
exists and 
\begin{equation}
\left\Vert \int_{s}^{t}A^{\varepsilon }P_{t-r}Q(u(r))dr\right\Vert \leq
CK_{2}\left( \sup_{t\in (0,T]}||A^{\delta }u_{t}|\right) \left( t-s\right)
^{1-\epsilon }\text{.}  \label{fe1}
\end{equation}
\end{proposition}

\begin{proof}
Let $t\in (0,T]$ and consider $f(s)=A^{\varepsilon }P_{t-s}Q(u(s))$. Then
for $s_{i}\in (0,t)$, $s_{1}>s_{2}$, one has 
\begin{eqnarray*}
||f(s_{1})-f(s_{2})|| &\leq &||A^{\varepsilon }P_{t-s_{1}}\left(
Qu(s_{1})-Qu(s_{2})\right) || \\
&&+||A^{\varepsilon }P_{t-s_{1}}\left( I-P_{s_{1}-s_{2}}\right) Qu(s_{2})||
\\
&\leq &C(t-s_{1})^{-\varepsilon }||Qu(s_{1})-Qu(s_{2})|| \\
&&+C(t-s_{1})^{-\varepsilon }||\left( I-P_{s_{1}-s_{2}}\right) Qu(s_{2})|| \\
&\leq &C(t-s_{1})^{-\varepsilon }K_{1}\left( 2\sup_{t\in (0,T]}||A^{\delta
}u_{t}|\right) ||A^{\delta }\left( u(s_{1})-u(s_{2})\right) || \\
&&+C(t-s_{1})^{-\varepsilon }||\left( I-P_{s_{1}-s_{2}}\right) Qu(s_{2})||
\end{eqnarray*}%
Since $A^{\delta }u\in C((0,T],X)$ so is $u$ and letting $s_{1}\rightarrow
s_{2}\in (0,t)$ the continuity of the semigroup and of $u$ give $%
\lim_{s_{1}\rightarrow s_{2}}||f(s_{1})-f(s_{2})||=0$, i.e. $f$ is
continuous on $(0,t)$. Moreover, 
\begin{eqnarray*}
||f(s)|| &\leq &C(t-s)^{-\varepsilon }||Qu(s)|| \\
&\leq &C(t-s)^{-\varepsilon }K_{2}(||A^{\delta }u(s)||) \\
&\leq &CK_{2}\left( \sup_{t\in (0,T]}||A^{\delta }u_{t}|\right)
(t-s)^{-\epsilon }\text{,}
\end{eqnarray*}

\noindent \noindent from which (\ref{fe1}) follows immediately.
\end{proof}

For $u\in \mathbb{H}_{T}$ we define 
\begin{equation}
\mathbb{L}u(t)=P_{t}u_{0}-\int_{0}^{t}A^{\tau
}P_{t-s}Q(u(s))ds+\int_{0}^{t}P_{t-s}F(u(s))dw_{s}  \label{mi-2}
\end{equation}%
where $w=(w_{t})$ is an $\alpha $-H\"{o}lder continuous path in $Z$, where $%
\alpha \in (\frac{1}{2},1]$ such that $\alpha +\tau <1$, so that $2\alpha >1$%
. With these constraints, and recalling that $\delta \leq \alpha $ we see
that from the previous proposition that 
\begin{equation*}
\sup_{t\in \lbrack 0,T]}||A^{\delta }\int_{0}^{t}A^{\tau
}P_{t-s}Q(u(s))ds||\leq CK_{2}(\beta )T^{1-\delta -\tau }
\end{equation*}%
and%
\begin{eqnarray*}
&&||\int_{0}^{t_{2}}A^{\tau }P_{t_{2}-s}Q(u(s))ds-\int_{0}^{t_{1}}A^{\tau
}P_{t_{1}-s}Q(u(s))ds|| \\
&=&\left\Vert \int_{0}^{t_{1}}\left( P_{t_{2}-t_{1}}-I\right) A^{\tau
}P_{t_{1}-s}Q\left( u\left( s\right) \right) ds+\int_{t_{1}}^{t_{2}}A^{\tau
}P_{t_{2}-s}Q\left( u\left( s\right) \right) ds\right\Vert \\
&\leq &C_{1}\left( t_{2}-t_{1}\right) ^{\alpha }K_{2}\left( \sup_{t\in \left[
0,T\right] }A^{\delta }u_{t}\right) t_{1}^{1-\left( \alpha +\tau \right)
}+C_{2}K_{2}\left( \sup_{t\in \left[ 0,T\right] }A^{\delta }u_{t}\right)
\left( t_{2}-t_{1}\right) ^{1-\tau } \\
&\leq &CT^{1-\left( \alpha +\tau \right) }K_{2}(\beta )\left(
t_{2}-t_{1}\right) ^{\alpha },
\end{eqnarray*}%
so that the Bochner integral appearing in (\ref{mi-2}) has the right
properties to be in $\mathbb{H}_{T}.$ We now turn our attention to the the
properties of the Young integral featuring in (\ref{mi-2}).

\section{The Young Integral}

We will use the Young integral to make a rigorous path-wise definition of
the stochastic integral above provided the process has sample paths of the
correct regularity. The example to keep in mind is fractional Brownian
motion. The following is due to L.C. Young from the 1930s when he made the
definition following investigations into the convergence of Fourier series
(see \cite{young}). It is an extension of Stieltjes integration to the case
of paths of finite $p$ variation for $p>1$.

\begin{definition}
A continuous path $f:\left[ 0,T\right] \rightarrow X$ \ where $X$ is a
Banach space is said to have finite $p$ variation on $\left[ 0,T\right] $ if 
\begin{equation*}
\left\Vert f\right\Vert _{p,\left[ 0,T\right] }:=\left[ \sup_{D}\sum_{k}%
\left\Vert f_{t_{k}}-f_{t_{k-1}}\right\Vert _{X}^{p}\right] ^{\frac{1}{p}%
}<\infty
\end{equation*}%
where by $\sup_{D}$ we understand supremum over all partitions of $\left[ 0,T%
\right] $.
\end{definition}

In analogy with the Riemann sums for Stieltjes integrals of paths of finite
variation, we have the following general definition of the integral in the
case of real valued functions $f$ and $w$ defined on the interval $\left[ 0,T%
\right] $.

\begin{remark}
It is obvious that if a path $f$ is H\"{o}lder continuous with exponent $%
h\in \left( 0,1\right) $ on the interval $\left[ s,t\right] $ then $f$ has
finite $\frac{1}{h}$ variation and%
\begin{equation*}
\left\Vert f\right\Vert _{\frac{1}{h},\left[ s,t\right] }\equiv \left(
\sup_{D\left[ s,t\right] }\sum_{k}\left\vert \left\vert
f_{t_{k}}-f_{t_{k-1}}\right\vert \right\vert ^{\frac{1}{h}}\right) ^{h}\leq
\left\Vert f\right\Vert _{h-H\ddot{o}l,\left[ s,t\right] }\left( t-s\right)
^{h}
\end{equation*}
We will use this fact several times in later proofs.
\end{remark}

\begin{definition}
The Stieltjes integral 
\begin{equation*}
\int_{0}^{T}f\left( t\right) dw_{t}
\end{equation*}%
is said to exist in the Riemann sense with value $I$ provided that for all $%
\delta >0$ there exists a $\varepsilon _{\delta }>0$ such that if the
partition $0=t_{0}<t_{1}<\cdots <t_{N-1}<t_{N}=T$ satisfies $\left\vert
t_{i}-t_{i-1}\right\vert <\delta $ for all $i$, then%
\begin{equation*}
\left\vert \sum_{k=0}^{N-1}f_{t_{k}}\left( w_{t_{k+1}}-w_{t_{k}}\right)
-I\right\vert <\varepsilon _{\delta }
\end{equation*}%
with $\varepsilon _{\delta }\rightarrow 0$ as $\delta \rightarrow 0$.
\end{definition}

\begin{theorem}
If the real valued functions $f$ and $w$ have finite $p$ and $q$ variation
respectively such that $p,q>0$ and $\frac{1}{p}+\frac{1}{q}>1$, then the
Stieltjes integral 
\begin{equation*}
\int_{0}^{T}f\left( t\right) dw_{t}
\end{equation*}%
exists in the Riemann sense.
\end{theorem}

In fact the definition of the integral and Young's theorem can be extended
to the infinite dimensional case when $w:\left[ 0,T\right] \rightarrow Z$
and $f:\left[ 0,T\right] \rightarrow L\left( Z,X\right) $ are bounded paths
with finite $p$ and $q$ variation respectively. In this case, the integral $%
\int_{0}^{t}f\left( s\right) dw_{s}:\left[ 0,T\right] \rightarrow X$ is a
bounded path of finite $q$ variation and we have the estimate 
\begin{equation*}
\left\Vert \int_{0}^{\cdot }\left( f\left( s\right) -f\left( 0\right)
\right) dw_{s}\right\Vert _{p,\left[ 0,T\right] }\leq C\left\Vert
f\right\Vert _{q,\left[ 0,T\right] }\left\Vert w\right\Vert _{p,\left[ 0,T%
\right] }
\end{equation*}%
known as Young's inequality. The details can be found in \cite{caruana}.
Using this, we can in a path-wise way define the integral $%
\int_{0}^{t}A^{\varepsilon }P_{t-s}F\left( u\left( s\right) \right) dw_{s}$
for processes $w$ with sufficiently regular sample paths.

\begin{lemma}
\bigskip Let $u\in \mathbb{H}_{T},$ $w:\left[ 0,T\right] \rightarrow Z$ be $%
\alpha -$H\"{o}lder continuous, $\epsilon \in \left[ 0,\alpha \right] $ and $%
\delta +\alpha >1$. If $F:X\rightarrow L_{2}\left( Z,X\right) $ satisfies
conditions (\ref{eq:cond1}) and (\ref{eq:cond2}) \noindent then 
\begin{equation}
\left\vert \left\vert A^{\epsilon }P_{t-s}F\left( u_{s}\right) -A^{\epsilon
}P_{t-r}F\left( u_{r}\right) \right\vert \right\vert _{L_{2}\left(
Z,X\right) }\leq C\left( \sup_{r\in (0,T]}\left\vert \left\vert A^{\delta
}u_{r}\right\vert \right\vert +T^{\alpha -\delta }+1\right) \left\vert
s-r\right\vert ^{\delta }  \label{holder bound}
\end{equation}

\noindent \noindent for all $r,s\in \left[ 0,t\right] ,t\in \left[ 0,T\right]
$. In particular, $A^{\epsilon }P_{t-\cdot }F\left( u_{\cdot }\right) :\left[
0,t\right] \rightarrow L_{2}\left( Z,X\right) $ is $\delta $-H\"{o}lder
continuous and hence the Young integral 
\begin{equation*}
\int_{0}^{t}A^{\epsilon }P_{t-s}F\left( u_{s}\right) dw_{s}
\end{equation*}

\noindent exists and satisfies 
\begin{equation}
\left\vert \left\vert \int_{0}^{t}A^{\epsilon }P_{t-s}F\left( u_{s}\right)
dw_{s}-A^{\epsilon }P_{t}F\left( u_{0}\right) \left( w_{t}-w_{0}\right)
\right\vert \right\vert _{X}\leq C\left( \sup_{r\in (0,T]}\left\vert
\left\vert A^{\delta }u_{r}\right\vert \right\vert +T^{\alpha -\delta
}+1\right) .  \label{young's inequality}
\end{equation}
\end{lemma}

\begin{proof}
For $0\leq r<s\leq t$ let us observe by using (\ref{eq:cond1})\ and (\ref%
{consequence2}) 
\begin{eqnarray*}
&&\left\vert \left\vert \left( A^{\epsilon }P_{t-s}F\left( u_{s}\right)
-A^{\epsilon }P_{t-r}F\left( u_{r}\right) \right) \right\vert \right\vert
_{L_{2}\left( Z,X\right) } \\
&\leq &C\left( \left\vert \left\vert \left( I-P_{s-r}\right) A^{\epsilon
}F\left( u_{r}\right) \right\vert \right\vert _{L_{2}\left( Z,X\right)
}+\left\vert \left\vert A^{\epsilon }\left[ F\left( u_{s})-F(u_{r}\right) %
\right] \right\vert \right\vert _{L_{2}\left( Z,X\right) }\right) \\
&\leq &\frac{C\left( \delta \right) }{\delta }\left( s-r\right) ^{\delta
}\left\Vert A^{\epsilon +\delta }F\left( u_{s}\right) \right\Vert
_{L_{2}\left( Z,X\right) }+C\left\vert \left\vert u_{s}-u_{r}\right\vert
\right\vert \\
&\leq &\left( s-r\right) ^{\delta }\left( C_{1}\left\vert \left\vert
A^{\delta }u_{s}\right\vert \right\vert +C_{2}\right) +C_{3}T^{\alpha
-\delta }\left( s-r\right) ^{\delta }.
\end{eqnarray*}

\noindent \noindent From which it follows that 
\begin{equation*}
\left\vert \left\vert \left( A^{\epsilon }P_{t-s}F\left( u_{s}\right)
-A^{\epsilon }P_{t-r}F\left( u_{r}\right) \right) \right\vert \right\vert
\leq C\left( \sup_{r\in (0,T]}\left\vert \left\vert A^{\delta
}u_{r}\right\vert \right\vert +T^{\alpha -\delta }+1\right) \left(
s-r\right) ^{\delta }.
\end{equation*}

\noindent Estimate$\ $(\ref{holder bound}) follows at once. From \cite%
{caruana} this is sufficient to guarantee the existence of the Young
integral and (\ref{young's inequality}) follows from Young's inequality.
\end{proof}

\section{\protect\bigskip The non-Linear mapping $\mathbb{L}$}

We now prove that the non-linear mapping $\mathbb{L}$ is a (non-linear)
bounded operator. More precisely:

\begin{theorem}
\label{invariance}If $u\in \mathbb{H}_{T}$ and $u_{0}\in D(A^{\alpha })$,
then $\mathbb{L}u\in \mathbb{H}_{T}$ and the following estimates hold: if $%
||u||_{\mathbb{H}_{T}}\leq \beta $ 
\begin{eqnarray*}
||\mathbb{L}u||_{\mathbb{H}_{T}} &\leq &C_{1}\left( 1+T^{\alpha }\right)
||A^{\alpha }u_{0}||+C_{2}\left( T^{1-\delta -\tau }+T^{1-\tau -\alpha
}\right) K_{2}(\beta ) \\
&&+C_{3}\left[ \beta +1\right] \left( T^{\alpha }+T^{\delta }+T^{\alpha
+\delta }+T^{2\alpha }\right) +C_{4}\left( 1+T^{\alpha }\right) \left\Vert
A^{\alpha }F\left( u_{0}\right) \right\Vert _{op}\text{.}
\end{eqnarray*}
\end{theorem}

\begin{proof}
There are three terms appearing in the definition of $\mathbb{L}$, namely $%
P_{t}u_{0}$, the ordinary integral 
\begin{equation*}
J_{t}=\int_{0}^{t}A^{\tau }P_{t-s}Q(u_{s})ds
\end{equation*}%
and the Young integral 
\begin{equation*}
U_{t}\equiv \int_{0}^{t}P_{t-s}F(u_{s})dw_{s}\text{.}
\end{equation*}%
Let us estimate their $\mathbb{H}_{T}$ norms one by one. Firstly, 
\begin{equation*}
||A^{\delta }P_{t}u_{0}||\leq ||A^{\delta }u_{0}||\leq \beta
\end{equation*}%
and 
\begin{eqnarray*}
||P_{t}u_{0}-P_{s}u_{0}|| &=&||\left( P_{t-s}-I\right) P_{s}u_{0}|| \\
&\leq &C_{\alpha }(t-s)^{\alpha }||A^{\alpha }P_{s}u_{0}|| \\
&\leq &C_{\alpha }||A^{\alpha }u_{0}||(t-s)^{\alpha }\text{.}
\end{eqnarray*}%
Secondly we consider the ordinary integral $J_{t}$. It is elementary that 
\begin{eqnarray*}
||A^{\delta }J_{t}|| &\leq &C_{1}\int_{0}^{t}(t-s)^{-(\delta +\tau
)}||Q(u_{s})||ds \\
&\leq &C_{1}\int_{0}^{t}(t-s)^{-(\delta +\tau )}K_{2}(||A^{\delta }u_{s}||)ds
\\
&\leq &C_{2}K_{2}(\beta )t^{1-\delta -\tau }\text{.}
\end{eqnarray*}%
To estimate the H\"{o}lder norm we use the following elementary formula 
\begin{equation*}
J_{t}-J_{s}=\int_{s}^{t}A^{\tau }P_{t-r}Q(u_{r})dr+\left( P_{t-s}-I\right)
\int_{0}^{s}A^{\tau }P_{s-r}Q(u_{r})dr\text{. }
\end{equation*}%
While it is easy to see that 
\begin{eqnarray*}
\left\Vert \int_{s}^{t}A^{\tau }P_{t-r}Q(u_{r})dr\right\Vert &\leq &\frac{%
C_{4}K_{2}(\beta )}{1-\tau }\left( t-s\right) ^{1-\tau } \\
&\leq &\frac{C_{4}K_{2}(\beta )}{1-\tau }\left( t-s\right) ^{\alpha }\left(
t-s\right) ^{1-\tau -\alpha } \\
&\leq &\frac{C_{4}K_{2}(\beta )}{1-\tau }T^{1-\tau -\alpha }\left(
t-s\right) ^{\alpha }.
\end{eqnarray*}

\noindent Using the bound 
\begin{equation*}
||x-P_{t}x||\leq \frac{C(\delta )}{\delta }t^{\delta }||A^{\delta }x||
\end{equation*}%
\noindent we deduce that%
\begin{eqnarray*}
\left\vert \left\vert \left( P_{t-s}-I\right) \int_{0}^{s}A^{\tau
}P_{s-r}Q\left( u_{r}\right) dr\right\vert \right\vert &\leq &\frac{C\left(
\alpha \right) }{\alpha }\left( t-s\right) ^{\alpha }\left\vert \left\vert
A^{\alpha }\int_{0}^{s}A^{\tau }P_{s-r}Q\left( u_{r}\right) dr\right\vert
\right\vert \\
&\leq &\frac{C\left( \alpha \right) }{\alpha \left( 1-\tau -\alpha \right) }%
K_{2}\left( \beta \right) s^{1-\tau -\alpha }\left( t-s\right) ^{\alpha } \\
&\leq &\frac{C\left( \alpha \right) }{\alpha \left( 1-\tau -\alpha \right) }%
K_{2}\left( \beta \right) T^{1-\tau -\alpha }\left( t-s\right) ^{\alpha }.
\end{eqnarray*}

\noindent Finally we handle the Young integral. \ Let $[s,t]\subset \lbrack
0,T]$ then 
\begin{eqnarray*}
U_{t}-U_{s}
&=&\int_{0}^{t}P_{t-r}F(u_{r})dw_{r}-\int_{0}^{s}P_{s-r}F(u_{r})dw_{r} \\
&=&\int_{s}^{t}P_{t-r}F(u_{r})dw_{r}+(P_{t-s}-I)%
\int_{0}^{s}P_{s-r}F(u_{r})dw_{r},
\end{eqnarray*}%
so that 
\begin{eqnarray}
||U_{t}-U_{s}|| &\leq &\left\Vert
\int_{s}^{t}P_{t-r}F(u_{r})dw_{r}\right\Vert +\left\Vert
(P_{t-s}-I)\int_{0}^{s}P_{s-r}F(u_{r})dw_{r}\right\Vert  \label{difference}
\\
&\leq &\left\Vert \int_{s}^{t}P_{t-r}F(u_{r})dw_{r}\right\Vert +\frac{%
C\left( \alpha \right) }{\alpha }\left( t-s\right) ^{\alpha }\left\Vert
\int_{0}^{s}A^{\alpha }P_{s-r}F(u_{r})dw_{r}\right\Vert  \notag
\end{eqnarray}%
To handle the two integrals, let us consider 
\begin{equation*}
R_{s,t}^{\epsilon }\equiv A^{\epsilon }\int_{s}^{t}P_{t-r}F(u_{r})dw_{r}
\end{equation*}%
for $\varepsilon \in \lbrack 0,\alpha ]$. \ From Young's inequality we have
that 
\begin{eqnarray*}
\left\vert \left\vert R_{s,t}^{\epsilon }\right\vert \right\vert
&=&\left\Vert \int_{s}^{t}A^{\varepsilon }P_{t-r}F(u_{r})dw_{r}\right\Vert \\
&\leq &C\left\vert \left\vert A^{\varepsilon }P_{t-\cdot }F(u_{\cdot
})\right\vert \right\vert _{\delta -\text{H\"{o}l}}\left\vert \left\vert
w\right\vert \right\vert _{\alpha -\text{H\"{o}l}}\left( t-s\right) ^{\delta
+\alpha }+\left\vert \left\vert A^{\varepsilon }P_{t-s}F(u_{s})\left(
w_{t}-w_{s}\right) \right\vert \right\vert \\
&\leq &C\left( \beta +1+T^{\alpha -\delta }\right) \left( t-s\right)
^{\delta +\alpha }+\left\vert \left\vert A^{\varepsilon
}P_{t-s}(F(u_{s})-F\left( u_{0}\right) )\left( w_{t}-w_{s}\right)
\right\vert \right\vert \\
&&+\left\vert \left\vert A^{\varepsilon }P_{t-s}F(u_{0})\left(
w_{t}-w_{s}\right) \right\vert \right\vert
\end{eqnarray*}%
Hence, using condition \ref{Fcond} and lemma \ref{Flemma} we can deduce 
\begin{equation}
\left\vert \left\vert R_{s,t}^{\epsilon }\right\vert \right\vert \leq
C\left( \beta +1+T^{\alpha -\delta }\right) \left( t-s\right) ^{\delta
+\alpha }+C\beta s^{\alpha }\left( t-s\right) ^{\alpha }+C\left\vert
\left\vert A^{\alpha }F\left( u_{0}\right) \right\vert \right\vert
_{op}\left( t-s\right) ^{\alpha }.  \label{integral bound}
\end{equation}

An application of this with $\epsilon =\delta $ yields 
\begin{eqnarray*}
\left\Vert A^{\delta }U_{t}\right\Vert &\leq &\left\Vert
\int_{0}^{t}A^{\delta }P_{t-r}F(u_{r})dw_{r}\right\Vert \\
&\leq &C\left( \beta +1+T^{\alpha -\delta }\right) T^{\delta +\alpha
}+C\beta T^{2\alpha }+C\left\vert \left\vert A^{\alpha }F\left( u_{0}\right)
\right\vert \right\vert _{op}T^{\alpha }.
\end{eqnarray*}%
\noindent Then, two further applications of (\ref{integral bound}) with $%
\epsilon =0$ and $\epsilon =\alpha $ show%
\begin{eqnarray}
\left\Vert \int_{s}^{t}P_{t-r}F(u_{r})dw_{r}\right\Vert &\leq &C\left( \beta
+1+T^{\alpha -\delta }\right) \left( t-s\right) ^{\delta +\alpha }+C\beta
T^{\alpha }\left( t-s\right) ^{\alpha }  \label{difference2} \\
&&+C\left\vert \left\vert A^{\alpha }F\left( u_{0}\right) \right\vert
\right\vert _{op}\left( t-s\right) ^{\alpha }  \notag \\
\left\Vert \int_{0}^{s}A^{\alpha }P_{s-r}F(u_{r})dw_{r}\right\Vert &\leq
&C\left( \beta +1+T^{\alpha -\delta }\right) T^{\delta +\alpha }+C\beta
T^{2\alpha }  \notag \\
&&+C\left\vert \left\vert A^{\alpha }F\left( u_{0}\right) \right\vert
\right\vert _{op}T^{\alpha }.
\end{eqnarray}

\noindent Assembling (\ref{difference}) and (\ref{difference2}) provides%
\begin{equation*}
\frac{||U_{t}-U_{s}||}{\left( t-s\right) ^{\alpha }}\leq C\left( \beta
+1+T^{\alpha -\delta }\right) \left( T^{\delta }+T^{\alpha +\delta }\right)
+C\beta T^{\alpha }\left( 1+T^{\alpha }\right) +C\left\vert \left\vert
A^{\alpha }F\left( u_{0}\right) \right\vert \right\vert _{op}\left(
1+T^{\alpha }\right) .
\end{equation*}

\noindent The result then follows by collecting together the appropriate
terms.
\end{proof}

Next we prove that $\mathbb{L}$ is locally Lipschitz.

\begin{theorem}
\label{contract}Let $u,v\in \mathbb{H}_{T}$ such that $||u||_{\mathbb{H}%
_{T}}\leq \beta $, $||v||_{\mathbb{H}_{T}}\leq \beta $, $u_{0}=v_{0}\in
D(A^{\alpha })$, then 
\begin{eqnarray*}
&&||\mathbb{L}u-\mathbb{L}v||_{\mathbb{H}_{T}} \\
&\leq &C\left[ \left( T^{1-\delta -\tau }+T^{1-\alpha -\tau }\right)
K_{1}\left( 2\beta \right) +\left( T^{\alpha }+T^{2\alpha }\right) \left(
\sup_{x\in X}\left\vert \left\vert DA^{\alpha }F\left( x\right) \right\vert
\right\vert +\frac{3}{2}\beta +1\right) \right] ||u-v||_{\mathbb{H}_{T}}.
\end{eqnarray*}
\end{theorem}

\begin{proof}
Let $u,v\in \mathbb{H}_{T}$ such that $||u||_{\mathbb{H}_{T}}\leq \beta $
and $||v||_{\mathbb{H}_{T}}\leq \beta $. Let $\theta _{t}=u_{t}-v_{t}$ and
consider 
\begin{equation*}
D_{t}=\int_{0}^{t}A^{\tau }P_{t-r}\left( Q(u_{r})-Q(v_{r})\right) dr\text{.}
\end{equation*}%
If $t>s$ then 
\begin{eqnarray*}
D_{t}-D_{s} &=&\int_{s}^{t}A^{\tau }P_{t-r}\left[ Q(u_{r})-Q(v_{r})\right] dr
\\
&&+\left( P_{t-s}-I\right) \int_{0}^{s}A^{\tau }P_{s-r}\left[
Q(u_{r})-Q(v_{r})\right] dr\text{. }
\end{eqnarray*}%
It follows that 
\begin{eqnarray*}
||D_{t}-D_{s}|| &\leq &C\int_{s}^{t}\left( t-r\right) ^{-\tau }\left\Vert
Q(u_{r})-Q(v_{r})\right\Vert dr \\
&&+C|t-s|^{\alpha }\int_{0}^{s}\left( s-r\right) ^{-\tau -\alpha }\left\Vert
Q(u_{r})-Q(v_{r})\right\Vert dr
\end{eqnarray*}%
together with the estimate 
\begin{equation}
||Q(u_{r})-Q(v_{r})||\leq K_{1}(2\beta )||A^{\delta }\theta _{r}||
\label{cr-02}
\end{equation}%
\noindent we obtain 
\begin{eqnarray*}
||D_{t}-D_{s}|| &\leq &CK_{1}(2\beta )\int_{s}^{t}\left( t-r\right) ^{-\tau
}||A^{\delta }\theta _{r}||dr \\
&&+CK_{1}(2\beta )|t-s|^{\alpha }\int_{0}^{s}\left( s-r\right) ^{-\tau
-\alpha }||A^{\delta }\theta _{r}||dr \\
&\leq &\frac{C}{1-\tau }K_{1}(2\beta )T^{1-\tau -\alpha }|t-s|^{\alpha
}||\theta ||_{\mathbb{H}_{T}} \\
&&+\frac{C}{1-\tau -\alpha }K_{1}(2\beta )T^{1-\tau -\alpha }||\theta ||_{%
\mathbb{H}_{T}}|t-s|^{\alpha }\text{.}
\end{eqnarray*}

\noindent Using the same argument we can obtain for all $t\in \left[ 0,T%
\right] $ 
\begin{equation*}
\left\vert \left\vert A^{\delta }D_{t}\right\vert \right\vert \leq \frac{%
T^{1-\delta -\tau }}{1-\delta -\tau }K_{1}\left( 2\beta \right) ||\theta ||_{%
\mathbb{H}_{T}}.
\end{equation*}

\noindent \noindent \noindent We now define 
\begin{equation*}
\Delta _{t}\equiv \int_{0}^{t}P_{t-r}\left[ F\left( u_{r}\right) -F\left(
v_{r}\right) \right] dw_{r}
\end{equation*}

\noindent \noindent and consider 
\begin{eqnarray*}
\Delta _{t}-\Delta _{s} &=&\int_{s}^{t}P_{t-r}\left[ F\left( u_{r}\right)
-F\left( v_{r}\right) \right] dw_{r} \\
&&+\left( P_{t-s}-I\right) \int_{0}^{s}P_{s-r}\left[ F\left( u_{r}\right)
-F\left( v_{r}\right) \right] dw_{r},
\end{eqnarray*}

\noindent \noindent \noindent which satisfies 
\begin{eqnarray*}
\left\vert \left\vert \Delta _{t}-\Delta _{s}\right\vert \right\vert &\leq
&\left\vert \left\vert \int_{s}^{t}P_{t-r}\left[ F\left( u_{r}\right)
-F\left( v_{r}\right) \right] dw_{r}\right\vert \right\vert \\
&&+C\left( t-s\right) ^{\alpha }\left\vert \left\vert \int_{0}^{s}A^{\alpha
}P_{s-r}\left[ F\left( u_{r}\right) -F\left( v_{r}\right) \right]
dw_{r}\right\vert \right\vert .
\end{eqnarray*}

\noindent \noindent \noindent As before, we consider for $\epsilon \in \left[
0,\alpha \right] $ 
\begin{equation*}
R_{s,t}^{\epsilon }\equiv \int_{s}^{t}A^{\varepsilon }P_{t-r}\left[ F\left(
u_{r}\right) -F\left( v_{r}\right) \right] dw_{r},
\end{equation*}

\noindent \noindent and note that Young's inequality gives 
\begin{eqnarray}
\left\vert \left\vert R_{s,t}^{\epsilon }\right\vert \right\vert &\leq
&C\left\vert \left\vert A^{\alpha }P_{t-\cdot }\left[ F\left( u_{\cdot
}\right) -F\left( v_{\cdot }\right) \right] \right\vert \right\vert _{\alpha
-\text{H\"{o}l}}\left\vert \left\vert w\right\vert \right\vert _{\alpha -%
\text{H\"{o}l}}\left( t-s\right) ^{2\alpha }  \notag \\
&&+\left\vert \left\vert A^{\alpha }P_{t-s}(F\left( u_{s}\right) -F\left(
v_{s}\right) )\left( w_{t}-w_{s}\right) \right\vert \right\vert  \notag \\
&\leq &C\left\vert \left\vert A^{\alpha }P_{t-\cdot }\left[ F\left( u_{\cdot
}\right) -F\left( v_{\cdot }\right) \right] \right\vert \right\vert _{\alpha
-\text{H\"{o}l}}\left\vert \left\vert w\right\vert \right\vert _{\alpha -%
\text{H\"{o}l}}\left( t-s\right) ^{2\alpha }  \notag \\
&&+C\left\vert \left\vert (u_{s}-u_{0})-(v_{s}-v_{0})\right\vert \right\vert
\left( t-s\right) ^{\alpha }.  \label{ineq1}
\end{eqnarray}

\noindent From lemma we have 
\begin{eqnarray}
&&\left\vert \left\vert A^{\alpha }\left[ F\left( u_{t}\right) -F\left(
v_{s}\right) \right] -A^{\alpha }\left[ F\left( u_{t}\right) -F\left(
v_{s}\right) \right] \right\vert \right\vert _{L_{2}\left( Z,X\right) }
\label{ineq2} \\
&\leq &\left( \sup_{x}\left\vert \left\vert DA^{\alpha }F\left( x\right)
\right\vert \right\vert +\frac{3}{2}\left\vert \left\vert v\right\vert
\right\vert _{\mathbb{H}_{T}}\right) ||\theta ||_{\mathbb{H}%
_{T}}(t-s)^{\alpha }  \notag
\end{eqnarray}

\noindent \noindent Combining (\ref{ineq1}) and (\ref{ineq2}) we can deduce 
\begin{equation*}
\frac{\left\vert \left\vert \Delta _{t}-\Delta _{s}\right\vert \right\vert }{%
\left( t-s\right) ^{\alpha }}\leq C\left( T^{\alpha }+T^{2\alpha }\right)
\left( \sup_{x}\left\vert \left\vert DA^{\alpha }F\left( x\right)
\right\vert \right\vert +\frac{3}{2}\beta +1\right) ||\theta ||_{\mathbb{H}%
_{T}},
\end{equation*}

\noindent and hence 
\begin{eqnarray*}
||Lu-Lv||_{\mathbb{H}_{T}} &\leq &\sup_{t\in \left[ 0,T\right] }\left\vert
\left\vert A^{\delta }D_{t}\right\vert \right\vert +\sup_{t\in \left[ 0,T%
\right] }\left\vert \left\vert A^{\delta }\Delta _{t}\right\vert \right\vert
\\
&&+\sup_{0\leq s<t\leq T}\frac{\left\vert \left\vert D_{t}-D_{s}\right\vert
\right\vert }{\left( t-s\right) ^{\alpha }}+\sup_{0\leq s<t\leq T}\frac{%
\left\vert \left\vert \Delta _{t}-\Delta _{s}\right\vert \right\vert }{%
\left( t-s\right) ^{\alpha }} \\
&\leq &C\left( T^{1-\delta -\tau }+T^{1-\alpha -\tau }\right) K_{1}\left(
2\beta \right) ||\theta ||_{\mathbb{H}_{T}} \\
&&+C\left( T^{\alpha }+T^{2\alpha }\right) \left( \sup_{x}\left\vert
\left\vert DA^{\alpha }F\left( x\right) \right\vert \right\vert +\frac{3}{2}%
\beta +1\right) ||\theta ||_{\mathbb{H}_{T}}.
\end{eqnarray*}
\end{proof}

\begin{theorem}
Let $u_{0}\in D(A^{\alpha })$ then for some $T^{\ast }>0$ depending only on $%
||A^{\alpha }u_{0}||$ and $\left\Vert A^{\alpha }F\left( u_{0}\right)
\right\Vert _{op}$ there is a unique $u\in \mathbb{H}_{T^{\ast }}$ such that 
\begin{equation*}
u_{t}=P_{t}u_{0}-\int_{0}^{t}A^{\tau
}P_{t-s}Q(u_{s})ds+\int_{0}^{t}P_{t-s}F(u_{s})dw_{s}\text{.}
\end{equation*}
\end{theorem}

\begin{proof}
By choosing $\beta =\beta \left( ||A^{\alpha }u_{0}||,\left\Vert A^{\alpha
}F\left( u_{0}\right) \right\Vert _{op}\right) $ sufficiently large Theorem %
\ref{invariance} allows us to ensure that

\noindent $\mathbb{L}\left( \mathbb{H}_{T}\cap \left\{ u:\left\vert
\left\vert u\right\vert \right\vert _{\mathbb{H}_{T}}\leq \beta \right\}
\right) \mathbb{\subseteq }\mathbb{H}_{T}\cap \left\{ u:\left\vert
\left\vert u\right\vert \right\vert _{\mathbb{H}_{T}}\leq \beta \right\} $
for all $T\in \left[ 0,T_{1}\right] $ and some $T_{1}>0.$ From Theorem \ref%
{contract} we see that $\mathbb{L}$ is a contraction on $\mathbb{H}_{T}\cap
\left\{ u:\left\vert \left\vert u\right\vert \right\vert _{\mathbb{H}%
_{T}}\leq \beta \right\} $ for all $T\in \lbrack 0,T_{2}]$ some $T_{2}>0$.
By taking $T^{\ast }=T_{1}\wedge T_{2}$ the result follows by a standard
contraction-mapping fixed point argument.
\end{proof}

\section{Randomly forced Navier-Stokes equations}

In particular this theorem applies to Navier-Stokes equation driven by
fractional Brownian motion with $h>\frac{1}{2}$. In this case, the equation
we want to solve is given together with the Dirichlet boundary conditions
and no-slip condition by%
\begin{equation*}
\frac{\partial }{\partial t}u+u.\nabla u=\Delta u-\nabla p+F(u)\dot{w},\text{
}\nabla .u=0,\text{ }u|_{\Gamma }=0
\end{equation*}

\noindent in some bounded domain $\Omega $ with compact, smooth boundary $%
\Gamma $. As is well known (see for instance \cite{ladyzhenskaya}) the
orthogonal complement of the set $K^{\infty }\left( \Omega \right) =\left\{
u\in C^{\infty }\left( \Omega \right) :\nabla .u=0\right\} $ in $L^{2}\left(
\Omega \right) $ is given by $G\left( \Omega \right) =\left\{ u:u=\nabla p,%
\text{ }p,\nabla p\in L^{2}\left( \Omega \right) \right\} $, and if we
denote the closure of $K^{\infty }\left( \Omega \right) $ in $L^{2}\left(
\Omega \right) $ by $K\left( \Omega \right) $ then we have the decomposition 
$L^{2}\left( \Omega \right) =K\left( \Omega \right) \oplus G\left( \Omega
\right) $. Letting $P_{\infty }:L^{2}\left( \Omega \right) \rightarrow
K\left( \Omega \right) $ be the orthogonal projection onto $K\left( \Omega
\right) $ the above equation can be written as 
\begin{equation*}
\frac{\partial }{\partial t}u+Au+P_{\infty }(u.\nabla u)=P_{\infty }F(u)\dot{%
w},
\end{equation*}%
where $A=-P_{\infty }\circ \Delta $ is the Stokes operator. More precisely
we can define the usual Stokes operator (i.e. with Dirichlet boundary
conditions) in the following manner.

\begin{definition}
The Stokes operator acting in $K\left( \Omega \right) $ is given by the
self-adjoint linear operator $-P_{\infty }\circ \Delta $ with domain $%
D\left( -P_{\infty }\circ \Delta \right) =\left\{ u:u\in W^{2,2}\left(
\Omega \right) ,\left. u\right\vert _{\partial D}=0\right\} $ where $%
\triangle $ is the usual trace Laplacian on functions (or vector fields)
where the derivative is taken in the generalized sense.
\end{definition}

The following result is well known and we refer the reader to \cite%
{constantin} for the details.

\begin{theorem}
The Stokes operator with Dirichlet boundary conditions is self-adjoint and
its inverse is a compact operator in $K\left( \Omega \right) $.
\end{theorem}

Hence, for the Dirichlet boundary condition case, the Stokes operator
satisfies the conditions required by our fixed point theorem. It is also
clear that if $w$ is a fractional Brownian motion with Hurst parameter $h>%
\frac{1}{2}$ then it satisfies the conditions required by the fixed point
theorem. It remains to verify the condition on $Q:=P_{\infty }\left( \left(
u\cdot \nabla \right) u\right) $.

\begin{lemma}
The operator $Q\left( u\right) =A^{\frac{1}{4}}P_{\infty }\left( \left(
u\cdot \nabla \right) u\right) $ is weakly Gateaux differentiable on $%
D\left( A^{\delta }\right) $ where $A$ is either the Stokes operator
associated with either the Dirichlet or Navier and kinematic boundary
conditions. Its derivative is given by%
\begin{equation*}
DQ\left( u\right) \xi =A^{\frac{1}{4}}P_{\infty }\left( \left( u\cdot \nabla
\right) \xi \right) +A^{\frac{1}{4}}P_{\infty }\left( \left( \xi \cdot
\nabla \right) u\right)
\end{equation*}
\end{lemma}

\begin{proof}
Let $z\in K\left( \Omega \right) ^{\ast }=K\left( \Omega \right) $, then it
is clear that 
\begin{align*}
\frac{d}{d\varepsilon }\left\langle A^{\frac{1}{4}}P_{\infty }^{{}}\left(
\left( \left[ u+\varepsilon \xi \right] \cdot \nabla \right) \left[
u+\varepsilon \xi \right] \right) ,z\right\rangle =& \frac{d}{d\varepsilon }%
\left( \left\langle A^{\frac{1}{4}}P_{\infty }\left( \left( u\cdot \nabla
\right) u\right) ,z\right\rangle \right. \\
& +\varepsilon \left\langle A^{\frac{1}{4}}P_{\infty }\left( \left( u\cdot
\nabla \right) \xi \right) ,z\right\rangle \\
& +\varepsilon \left\langle A^{\frac{1}{4}}P_{\infty }\left( \left( \xi
\cdot \nabla \right) u\right) ,z\right\rangle \\
& \left. +\varepsilon ^{2}\left\langle A^{\frac{1}{4}}P_{\infty }\left(
\left( \xi \cdot \nabla \right) \xi \right) ,z\right\rangle \right)
\end{align*}%
which implies 
\begin{equation*}
\left. \frac{d}{d\varepsilon }\left\langle A^{\frac{1}{4}}P_{\infty }\left(
\left( \left[ u+\varepsilon \xi \right] \cdot \nabla \right) \left[
u+\varepsilon \xi \right] \right) ,z\right\rangle \right\vert _{\varepsilon
=0}=\left\langle A^{\frac{1}{4}}P_{\infty }\left( \left( u\cdot \nabla
\right) \xi \right) +A^{\frac{1}{4}}P_{\infty }\left( \left( \xi \cdot
\nabla \right) u\right) ,z\right\rangle .
\end{equation*}
\end{proof}

Then the Sobolev inequality (\ref{eq:sobiq}) implies the relative
boundedness conditions on $Q$ hold in the case of Dirichlet boundary
conditions and so we can apply our fixed point theorem to this case.

\bigskip

\begin{acknowledgement}
\bigskip This research was gratefully supported in part by EPSRC grants
EP/F029578/1 and EP/P502667/1. The first named author would like to thank
Daniel Babarto for related discussions.
\end{acknowledgement}


\begin{thebibliography}{99}
\bibitem{caruana} Caruana, M; L\'{e}vy, T; Lyons, T.J., Differential
Equations Driven by Rough Paths. Lecture Notes in Mathematics, Springer 2007.

\bibitem{Chen} Chen, G., Qian, Z., A Study of the Navier-Stokes Equations
with the Kinematic and Navier Boundary Conditions (Preprint)

\bibitem{constantin} Constantin, P., Foias, C., Navier-Stokes Equations.
Chicago Lecture Notes in Mathematics, The University of Chicago Press 1988.

\bibitem{fujita} Fujita, H., Kato, T., On the Navier-Stokes Initial Value
Problem I. \textit{Arch. Rat. Mech. Anal. }\textbf{16 }(1964) No. 4, 269-315.

\bibitem{gubinelli} Gubinelli, M., Lejay, A., Tindel, S., Young Integrals
and SPDEs. \textit{Potential Analysis }\textbf{25 }(2006) No. 4, 307-326.

\bibitem{lyons1} Lyons, T., Differential Equations Driven by Rough Signals
(I): An extension of an Inequality of L.C. Young. \textit{Math. Res. Lett. }%
\textbf{1 }(1994) 451-464.

\bibitem{lyons2} Lyons, T., The Interpretation and Solution of Ordinary
Differential Equations Driven by Rough Signals. \textit{Proc. Symp. Pure
Math. }\textbf{57 }(1995) 115-128.

\bibitem{ladyzhenskaya} Ladyzhenskaya, O.A. \ Sixth Problem of the
Millennium : Navier-Stokes Equations, Existence and Smoothness, \textit{%
Russian Math. Surveys,} \textbf{58} No.2, 251-286

\bibitem{maslowski} Maslowski, NB., Nualart, D., Evolution Equations Driven
by a Fractional Brownian Motion. \textit{J. Funct. Anal. }\textbf{202 }%
(2003) 277-305.

\bibitem{kato1} Kato, T. Note on Fractional Powers of Linear Operators. 
\textit{Proc. Japan Acad.} \textbf{36} (1960) 94--96

\bibitem{kato2} Kato, T. Fractional Powers of Dissipative Operators. \textit{%
J. Math. Soc. Japan} \textbf{13} (1961) 246--274.

\bibitem{komatsu1} Komatsu, H. Fractional Powers of Operators. \textit{%
Pacific J. Math}. \textbf{19} (1966) 285--346

\bibitem{pazy} Pazy, A. Semigroups of Linear Opeators and Applications to
Partial Differential Equations, Springer 1983

\bibitem{phillips} Phillips, R. S. On the Generation of Semigroups of Linear
Operators.\textit{\ Pacific J. Math.} \textbf{2} (1952).

\bibitem{sobolev} Sobolevskii, P.E., On Non-Stationary Equations of
Hydrodynamics for Viscous Fluid. \textit{Doklady Akad. Nauk USSR }\textbf{%
128 }(1959) 45-48.

\bibitem{tindel} Tindel, S., Tudor, C.A., Viens, F., Stochastic Evolution
Equations with Fractional Brownian Motion. \textit{Prob. Theory Rel. Fields }%
\textbf{127 }(2003) 186-204.

\bibitem{young} Young, L.C., An Inequality of Holder Type, Connected with
Stieltjes Integration. \textit{Acta Math. }\textbf{67 }(1936) No. 1, 251-282.
\end{thebibliography}
\end{document}